\documentclass[pdflatex,sn-mathphys-num]{sn-jnl}% Math and Physical Sciences Numbered Reference Style
\usepackage{graphicx}%
\usepackage{multirow}%
\usepackage{amsmath,amssymb,amsfonts}%
\usepackage{amsthm}%
\usepackage{mathrsfs}%
\usepackage[title]{appendix}%
\usepackage{xcolor}%
\usepackage{textcomp}%
\usepackage{manyfoot}%
\usepackage{booktabs}%
\usepackage{algorithm}%
\usepackage{algorithmicx}%
\usepackage{algpseudocode}%
\usepackage{listings}%
\usepackage{bbm}

\theoremstyle{thmstyleone}%
\newtheorem{theorem}{Theorem}%  meant for continuous numbers
\newtheorem{proposition}[theorem]{Proposition}% 
\newtheorem{lemma}[theorem]{Lemma}% 
\newtheorem{corollary}[theorem]{Corollary}% 

\theoremstyle{thmstyletwo}%
\newtheorem{example}{Example}%
\newtheorem{remark}{Remark}%

\theoremstyle{thmstylethree}%
\newtheorem{definition}{Definition}%

\begin{document}

\title[Condensation in IPS with size-dependence]{Condensation in stochastic lattice gases with size-dependent stationary weights}


\author[1]{\fnm{Joshua} \sur{Blank}}
\author[2]{\fnm{Paul} \sur{Chleboun}}
\author[1]{\fnm{Stefan} \sur{Grosskinsky}}
\author*[3]{\fnm{Watthanan} \sur{Jatuviriyapornchai}\email{watthanan.jat@mahidol.ac.th}}


\affil[1]{\orgdiv{Institute of Mathematics}, \orgname{University of Augsburg}, \orgaddress{
% \street{Street},
\city{Augsburg},
\postcode{86135}, 
% \state{State}, 
\country{Germany}}}

\affil[2]{\orgdiv{Department of Statistics}, \orgname{University of Warwick}, \orgaddress{\city{Coventry}, \postcode{CV4 7AL}, \country{UK}}}

\affil*[3]{\orgdiv{Department of Mathematics, Faculty of Science}, \orgname{Mahidol University}, \orgaddress{\city{Bangkok}, \postcode{10400}, \country{Thailand}}}


\abstract{
	We consider stochastic lattice gases with stationary product weights and a polynomial perturbation vanishing with the system size that leads to condensation. If the density of particles exceeds a critical value the system phase separates into a bulk with homogeneous distribution of particles and a condensed phase. Depending on parameter values, the latter consists of a single macroscopic cluster or a diverging number of independent clusters on a smaller scale. We establish the condensation transition via the equivalence of ensembles and the main novelty is a derivation of the cluster size distribution using size-biased sampling, generalizing previous work on zero-range and inclusion processes. 
Simulations of zero-range processes illustrate our theoretical results on the condensate scale and size distribution.}

%%================================%%
%% Sample for structured abstract %%
%%================================%%




\maketitle

\section{Introduction}
\label{sec_introduction}

Condensation phenomena in interacting particle systems have been a topic of recent research interest. We focus on stochastic lattice gases where the number of particles is conserved and which have stationary product measures. A general class of such systems with an arbitrary number of particles per site has been introduced in \cite{bib_C85} and further generalized in \cite{Saada2016}. Various models in that class have been studied in the context of condensation, where a non-zero fraction of the total mass concentrates in vanishing volume fraction \cite{bib_CG13, bib_EW14, bib_G19}. Examples include the zero-range process \cite{bib_DGC98, bib_E00, bib_G03} with rigorous results in \cite{bib_JMP00, bib_GSS03, bib_AL09, bib_AGL13}, or the inclusion process \cite{grosskinsky2011condensation,grosskinsky2013dynamics} and generalizations of the latter \cite{bib_WE12,chau2015explosive} with heuristic results on instantaneous condensate formation. Of particular recent interest have been rigorous results on the dynamics of condensate formation \cite{Jara2025coarsening,beltran2026} and the stationary dynamics of the condensate in the context of metastability (see e.g.\ \cite{dommers,kim2021condensation,kim2025hierarchical} for the inclusion process). Many important results in this direction have been obtained in pioneering work by Claudio Landim and co-authors, establishing a martingale \cite{Landim2015martingale} and a resolvent approach \cite{landim2023resolvent} to metastability (see also \cite{claudiosurvey} and references therein). It is a great honour to contribute to this volume celebrating his achievements. Happy birthday Claudio!\\

In this contribution, we focus on stochastic lattice gases with stationary product measures that have a size-dependent perturbation leading to a condensation transition, generalizing recent work on zero-range \cite{JG2024ZRPfastrate} and inclusion processes \cite{bib_JCG20,gan2025poisson}. Such systems exhibit a great variety of possible cluster scales and hierarchical structures, while most rigorous results on classical models of condensing lattice gases mentioned above focus on systems with a single condensate at stationarity. This paper is a first step towards a systematic understanding of condensation transitions introduced by perturbations of stationary product weights that vanish with increasing system size. We can cover cases with a generic polynomial behaviour in the leading order term of the perturbation and establish the condensation transition in the context of the equivalence of ensembles. Depending on the strength of the perturbation, the condensed phase consists of a single cluster or many independent clusters on a smaller scale. In addition to the scale, we also derive the cluster-size distributions, which turn out to be of Gamma-type due to the polynomial form of the perturbation. We identify the transition from a single macroscopic to mesoscopic clusters as a function of the two main parameter values. On the transition line, the system exhibits an interesting hierarchical structure with many macroscopic clusters, which has been identified as the Poisson-Dirichlet distribution for a particular class of models corresponding to the weights of inclusion processes \cite{bib_CGG22,gan2025poisson}. The analysis of the general behaviour of our model class on the transition line is beyond the scope of this paper and left for future research.\\

% we send $N,L \to \infty$ while the density of particles per lattice size is converging to a constant $N/L \to \rho$. This is known as the thermodynamic limit. When the limiting density $\rho$ exceeds a critical density, $\rho > \rho_c$, we observe clusters occupying a vanishing fraction of sites. A cluster is site with a diverging number of particles. The remaining sites are in the bulk phase, sites with density $\rho_c$. Otherwise, if the density is subcritical $\rho < \rho_c$, we only observe the bulk i.e. sites with finite occupation numbers. 

The paper is organized as follows: In the next Section \ref{sec2} we introduce the mathematical setting, give our main results on the equivalence of ensembles (Theorem \ref{equivalence}) and the cluster size distribution (Theorems \ref{cluster} and \ref{thm3}), as well as examples of lattice gases that are covered by our general approach. The proofs of the main results are given in Section \ref{sec_proof}. They mostly rely on an asymptotic analysis of the canonical partition function using local limit theorems, in order to derive the scaling limit of the size-biase distribution function for occupation numbers.

%The thesis is organized as follows: In Section \ref{sec_general_definitions}, we introduce a rigorous mathematical framework for lattice gases and show the condition under which such processes have stationary measures with product structure. Section \ref{sec_fast_rates} shows some ZRP results published initially in \cite{bib_JG24}. These results are relevant for this work since they provide a blueprint for the theorems we later show for general IPS. This section also introduces the size-biased reordering of $\eta$, which is important when studying cluster size distributions. The weights of the lattice gas we study and some heuristics for the later rigorously proven results are presented in Section \ref{sec_mathematical_framework_and_heuristic}. Finally, in Section \ref{sec_main_results}, we show the main results of the thesis. We prove the equivalence of ensembles in Theorem \ref{theorem_equivalence_of_ensembles} and the distribution of the scaled cluster sizes in Theorem \ref{theorem_scaled_cluster_distribution}.

\section{The model and main results\label{sec2}}

\subsection{Definitions}

We consider stochastic lattice gases on the state space $E_L =\mathbb{N}_0^\Lambda$, where the lattice $\Lambda$ is a finite index set of size $|\Lambda |=L$. Particle configurations in $E_L$ are denoted by $\eta =(\eta_x :x\in\Lambda)$, and the number of particles $N_\Lambda (\eta )=\sum\limits_{x\in\Lambda} \eta_x$ is the only conserved quantity of the system. We assume that the resulting canonical stationary measures $\pi_{L,N}$ on the restricted finite state spaces
\[
    E_{L,N} = \{\eta \in \mathbb{N}_0^L: \sum_{x \in \Lambda} \eta_x = N\}\,,\quad\text{with }N\geq 0\,,
\]
are spatially homogeneous and of product form. That is, for all $\eta\in E_{L,N}$ we have
\begin{equation}\label{canon}
\pi_{L,N} [\eta ]=\frac{1}{Z_{L,N}} \prod_{x\in\Lambda} w_L (\eta_x )\quad\mbox{where}\quad Z_{L,N} =\sum_{\eta\in E_{L,N}} \prod_{x\in\Lambda} w_L (\eta_x )
\end{equation}
is the normalizing partition function. We will assume that the stationary weights $w_L$ are strictly positive and of the form
\begin{equation}\label{weights}
    w_L (n)=w(n)+\frac{\theta}{L^\gamma} (n+1)^\kappa \big( 1+\delta_L (n)\big)\quad\mbox{with }\theta 
    >0,\gamma >0\mbox{ and }\kappa >-1\ ,
\end{equation}
consisting of a size-independent (bulk) part $w$ and a perturbation. The latter decays asymptotically like a power-law, where we assume $L\sup_{n\leq N} \delta_L (n)\to 0$ in the thermodynamic limit \eqref{thermo}. We further assume that $w$ has exponential moments such that
\begin{equation}\label{exmoments}
    z(\phi ):=\sum_{n\geq 0} w(n)\phi^n <\infty\quad\mbox{if and only if}\quad\phi \in [0,\bar\phi)\mbox{ with }1<\bar\phi \leq\infty\ ,
\end{equation}
and that w.l.o.g.
\begin{equation}
\sum_{n\geq 0} w(n)=1\quad\mbox{with first moment}\quad \rho_c :=\sum_{n\geq 0} n\, w(n)<\infty\ .
\label{wmoments}    
\end{equation}
This structure is quite generic and includes a large class of lattice gases that arise in various applications, as is explained in Section \ref{sec_ips}. Due to the product structure of the stationary weights, there are also grand-canonical stationary measures $\nu^L_\phi [\eta ]= \bigotimes_{x \in \Lambda} \nu^1_{\phi} [\eta_x ]$ for all $\eta\in E_L$, with marginals
\begin{equation}
    \nu^1_{\phi} [\eta_x ]=\frac{1}{z_L (\phi )} w_L (\eta_x )\phi^{\eta_x}\quad\mbox{where}\quad z_L (\phi )=\sum_{n\geq 0} w_L (n)\phi^n
    \label{eq_grand_canonical_ensemble}
\end{equation}
is the normalizing partition function. Since we assume exponential moments \eqref{wmoments} and $\kappa >-1$ these measures are well defined for all $\phi\in [0,1)$. Note that for simplicity of notation we identify the measures $\pi_{L,N}$ and $\nu_\phi^L$ with their mass functions on the discrete state spaces $E_L$ and $E_{L,N}$.

The expected particle density under the grand-canonical measures is denoted by
\begin{equation}\label{rl}
    R_L (\phi ):=\langle\nu_\phi^1 ,\eta_x \rangle :=\frac{1}{z_L (\phi )}\sum_{n\geq 0} n\, w_L (n)\phi^n \ ,
\end{equation}
where we use the notation $\langle\mu ,f\rangle$ for the expectation of a function $f$ w.r.t.\ the distribution $\mu$. Since the power-law perturbation tilted by $\phi^n$ is summable and vanishes as $L\to\infty$, we have for all $\phi <1$
\begin{equation}\label{zrlim}
z_L (\phi )\to z(\phi ):=\sum_{n\geq 0} w(n)\phi^n \quad\mbox{and}\quad R_L (\phi )\to R(\phi ):=\frac{1}{z(\phi )}\sum_{n\geq 0} n\, w(n)\phi^n \ .    
\end{equation}
Note that $R(\phi )\nearrow\rho_c$ as $\phi\to 1$ approaches the critical density as defined in \eqref{wmoments}, whereas $R_L (\phi )\nearrow\infty$ as $\phi\to 1$ for all $L$. So the limits $\phi\nearrow 1$ and $L\to\infty$ do not commute, which leads to a condensation transition with critical density $\rho_c$ as explained in the next subsection. Both, $R(\phi )$ and $R_L (\phi )$, are also monotone increasing and continuous with $R(0)=R_L (0)=0$. 
Note that \eqref{exmoments} implies that the bulk weights $w$ can be tilted to any density, i.e.
\begin{equation}\label{tilt}
\mbox{for all }\rho\geq 0\mbox{ there exists }\phi <\bar\phi \mbox{ such that }R(\phi )=\rho\ .    
\end{equation}
% and we further assume that
% \[
% R(\phi )\nearrow\infty\quad\mbox{as }\phi\nearrow\bar\phi >1\ ,\quad\mbox{the radius of convergence of }z(\phi )\ .
% \]

\bigskip
\subsection{Main results}

We examine the asymptotic behaviour of the canonical measures $\pi_{L,N}$ in the
\begin{equation}\label{thermo}
    \mbox{thermodynamic limit}\quad L,N=N_L \to\infty\quad\mbox{with}\quad N/L\to\rho\geq 0\ .
\end{equation}
In our first result, we establish the condensation transition with critical density $\rho_c$ \eqref{wmoments}.

\begin{theorem}[\textbf{Equivalence of ensembles}]\label{equivalence}
Consider a lattice gas with stationary product weights $w_L$ as defined in \eqref{weights} and \eqref{wmoments}. Then for any finite set of sites $\Delta\subset\Lambda$ (for $L$ large enough), denoting $\pi_{L,N}^\Delta$ the marginal of $\pi_{L,N}$ on $\Delta$, we have in the thermodynamic limit \eqref{thermo} with $N/L\to\rho$
\[
\pi_{L,N}^\Delta \overset{d}{\longrightarrow} \nu_{\Phi (\rho )}^\Delta \quad\mbox{where}\quad \Phi (\rho )=\begin{cases}
	\phi &\text{for } R(\phi) = \rho \leq \rho_c\,, \\
	1 &\text{for } \rho > \rho_c\,,
\end{cases}
\]
where $\rho_c$ and $R(\phi)$ are given in \eqref{wmoments} and \eqref{zrlim}, respectively.
\end{theorem}

%\noindent This result implies that for densities $N/L\to\rho >\rho_c$ any finite dimensional marginal of the limiting distribution is given by the product measure $\nu_1^\Delta$ with density $\rho_c$. So a non-zero fraction $\rho -\rho_c$ of mass has to concentrate in a vanishing volume fraction.\\
\begin{proof} 
This follows from the same proof strategy as presented in \cite[Proposition A.1]{bib_CGG22}. It is a corollary of the local limit theorem for triangular arrays \cite{bib_DM95} and a large deviation estimate. It is sufficient to confirm that for each $\rho$ there exists a sequence $\phi_L \to \Phi(\rho)$ such that
\begin{align}
    \label{eq:weightsconv}
    \Big| \sum_{n\geq 0} n^2 \left( w_L (n) - w(n) \right) \phi_L^n \Big| \to 0 \quad \textrm{as } L \to \infty\,,
\end{align}
where the limiting weights $w(n)$ satisfy
\begin{itemize}\setlength{\itemindent}{2em}
	\item[(i)] $\sum_{n \geq 0} w(n) =1$\ ,\qquad\ \ (ii) $w(0) > 0$ and $\sup_{n} \left\{w(n-1) \wedge w(n)\right\} > 0$\ ,
	\item[(iii)] $\sum_{n \geq 0} n^2 w(n) < \infty$\ ,\quad (iv) $\frac{1}{L} \log w_L (a L) \to 0 \text{ for all } a>0$\ .
\end{itemize}
(i) - (iii) hold directly by the assumptions on $w(n)$. Since the limit $w$ decays exponentially, $w_L (aL )\simeq \theta (aL)^\kappa /L^\gamma $ is dominated by the perturbation as $L\to\infty$, so that (iv) holds. 

In particular, \eqref{eq:weightsconv} implies that the first two moments of the size dependent marginal converge as 
\begin{align}
    \label{eq:momsconv}
    \frac{1}{z_L (\phi )}\sum_{n\geq 0} n^p\, w_L (n)\phi_L^n \to \sum_{n\geq 0} n^p\, w (n)\Phi(\rho)^n  \ ,\quad \textrm{for } p \in \{0,1,2\}\,.
\end{align}
Furthermore, \eqref{eq:weightsconv} implies a Lindeberg condition for the centered and normalized sum $\sum_{x=1}^L\eta_{x,L}$, where $\eta_{x,L}\overset{\text{i.i.d.}}{\sim}  \nu_{\phi_L}^1$ (see the proof of \cite[Lemma A.4]{bib_CGG22}).
This in turn, with the lattice condition (ii) above, implies the required local limit theorem.

The only significant change from the proof of \cite[Proposition A.1]{bib_CGG22} is that \eqref{eq:weightsconv} can be confirmed as follows.
We choose $\phi_L = \Phi(\rho)$ for $\rho < \rho_c$ and $\phi_L= 1-L^{\frac{-\gamma}{\kappa + 4}}$ for $\rho \geq \rho_c$, so that
\begin{align*}
    \Big| \sum_{n\geq 0} n^2 \left( w_L (n) - w(n) \right) \phi_L^n \Big| &\leq \frac{\theta}{L^\gamma}\Big| \sum_{n\geq 0} (n+1)^{\kappa +2} \phi_L^n\Big| \\ 
    &\leq \frac{\theta\Gamma(\kappa+1)}{L^{\gamma}(1-\phi_L)^{\kappa +3}} \to 0 \quad \textrm{as } L \to \infty\,.
\end{align*}
The proof concludes following exactly  \cite{bib_CGG22} by a standard relative entropy argument, a large deviation upper bound for $\rho>\rho_c$ using (iv), and finally applying Pinsker's inequality.
% \textcolor{red}{PC: More/less detail as we chose OR we add an appendix}.
\end{proof}

Theorem \ref{equivalence} is equivalent to weak convergence for bounded, continuous cylinder functions $f\in C_b (E_\Delta )$\vspace*{-5mm}
\begin{equation*}
	\langle\pi_{L,N} ,f\rangle \to
	\begin{cases}
		\langle \nu_{\phi}^\Delta ,f\rangle &\text{for } R(\phi) = \rho \leq \rho_c \ , \\
		\langle \nu_1^\Delta ,f\rangle =\langle w^\Delta ,f\rangle &\text{for } \rho > \rho_c \ .
	\end{cases}
\end{equation*}
It implies that for $N/L\to\rho >\rho_c$ the system asymptotically phase separates into a homogeneous bulk with density $\rho_c$ and a condensate, where a macroscopic fraction  $(\rho -\rho_c )$ of mass concentrates on a vanishing volume fraction, forming clusters of diverging size. The equivalence of ensembles only describes the distribution of the bulk sites and in order to study the the condensate we use size-biased configurations $\tilde\eta :=(\eta_{\sigma (x)}:x\in\Lambda)\in E_{L,N}$ under $\pi_{L,N}$. The lattice indices are re-ordered according to the random permutation $\sigma :\Lambda\to\Lambda$, defined recursively as
\begin{align}
\sigma (1)&=x\mbox{ with prob. }\frac{\eta_x}{N}\ ,\quad x\in\Lambda\nonumber\\
\sigma (k+1)&=x\mbox{ with prob. }\frac{\eta_x}{N-(\eta_{\sigma (1)}+\ldots +\eta_{\sigma (k)})}\ ,\quad x\in\Lambda\setminus\big\{\sigma (1),\ldots ,\sigma (k)\big\}\ .\label{sizebias}
\end{align}
Equivalently, successive entries of $\tilde\eta$ are obtained by choosing a particle uniformly at random from the remaining sites and recording the occupation number at its location. This way cluster sites from the condensate appear with probability $(\rho -\rho_c )/\rho$ according to their mass fraction. 
Note that due to the product form of the distributions, we do not lose any spatial information under the size-biased resampling. We can now formulate our main result.\vspace*{-3mm}

\begin{theorem}[\textbf{Distribution of the condensate I}]\label{cluster}
Consider a lattice gas with stationary product weights $w_L$ as defined in \eqref{weights} and \eqref{wmoments} with $\kappa +2>\gamma \wedge 1$. 
% $\kappa >\min (-1,\gamma -2)$. 
% and $\gamma <k+2$. 
In the thermodynamic limit \eqref{thermo} with density $\rho > \rho_c$ \eqref{wmoments} we have, for a size-biased configuration \eqref{sizebias},
\begin{equation*}
\tilde{\eta}_1/C_L \overset{d}{\longrightarrow} 
\begin{cases}
    0 & \text{with probability } \rho_c/\rho \\
    Z & \text{with probability } (\rho-\rho_c)/\rho
\end{cases}\ ,
\end{equation*}
where $C_L = \left(\frac{(\rho-\rho_c) L^\gamma}{\theta \Gamma(\kappa+2)}\right)^{1 /(\kappa +2)} \ll L$ is the typical cluster scale and $Z \sim \Gamma(\kappa +2, 1)$ the random size on that scale. 
Additionally, for any fixed $m \in \mathbb{N},\ \tilde{\eta}_1, \dots, \tilde{\eta}_m$ converge to i.i.d. random variables with marginal law as above.
\end{theorem}
\vspace*{-3mm}
There is a simple heuristic for the cluster scale $C_L$. Assuming that the size-biased stationary weights converge to a non-degenerate density on scale $C_L \gg 1$, this implies with \eqref{weights} and finite exponential moments of $w(n)$ that
\[
C_L^2 w_L (C_L )\simeq \underbrace{C_L^2 w(C_L )}_{\to 0} +\theta C_L^{\kappa +2} /L^\gamma =O(1)
\]
so that $C_L \propto L^{\gamma /(\kappa +2)}$. 
For the canonical measure \eqref{canon} of a size-biased configuration $\tilde\eta$ \eqref{sizebias} the product structure leads to an explicit formula
\begin{align}
\pi_{L,N} [\tilde\eta_1 =n_1 ,\ldots ,\tilde\eta_m =n_m ]=\frac{L}{N} n_1 w_L (n_1 )\cdots &\frac{L-m+1}{N-n_1 -..-n_{m-1}} n_m w_L (n_m )\nonumber\\
&\times\frac{Z_{L-m,N-n_1 -\ldots-n_m}}{Z_{L,N}}
\label{sbformula}
\end{align}
for all $n_1 ,\ldots ,n_m \geq 0$ with $n_1 +\ldots +n_m \leq N$. 
The proof of Theorem \ref{cluster} then involves a careful asymptotic analysis of the ratio of partition functions, and will be presented in Section \ref{sec_proof}. 
The condition $\gamma <\kappa +2$ implies $C_L \ll L$, so each cluster only takes up a vanishing volume fraction and their sizes are asymptotically independent. This changes for $\gamma >\kappa +2$ where the condensed phase is dominated by a single cluster.

\begin{theorem}[\textbf{Distribution of the condensate II}]\label{thm3}
	Consider a lattice gas with stationary product weights $w_L$ as defined in \eqref{weights} and \eqref{wmoments} with $\gamma >\kappa +2 >1$. 
	In the thermodynamic limit \eqref{thermo} with density $\rho > \rho_c$ \eqref{wmoments} we have 
	\[
	\frac1L\max_{x\in\Lambda} \eta_x \overset{d}{\longrightarrow} \rho -\rho_c \ ,
	\]
	i.e. all the condensed mass concentrates in a single macroscopic cluster. Equivalently,
    \begin{equation*}
\tilde{\eta}_1/L \overset{d}{\longrightarrow} 
\begin{cases}
    0 & \text{with probability } \rho_c/\rho\,, \\
    \rho -\rho_c & \text{with probability } (\rho-\rho_c)/\rho\,.
\end{cases}
\end{equation*}
\end{theorem}

Our main results are illustrated Figure \ref{fig:phasediagram}. 
For the boundary case $\gamma =\kappa +2$ we expect several macroscopic clusters with a hierarchical distribution. This has been identified as a Poisson-Dirichlet distribution only for a certain class of lattice gases with $\gamma =1$ and $\kappa=-1$ \cite{bib_CGG22} which includes the inclusion process \cite{bib_JCG20}. Rigorous results shown in red are expected to extend to $\kappa <-1$, which would be consistent with classical results on condensation with a single cluster in \cite{bib_JMP00, bib_GSS03, bib_AL09, bib_AGL13} for $\kappa <-2$.

\begin{figure}
    \centering
\includegraphics[width=0.5\linewidth]{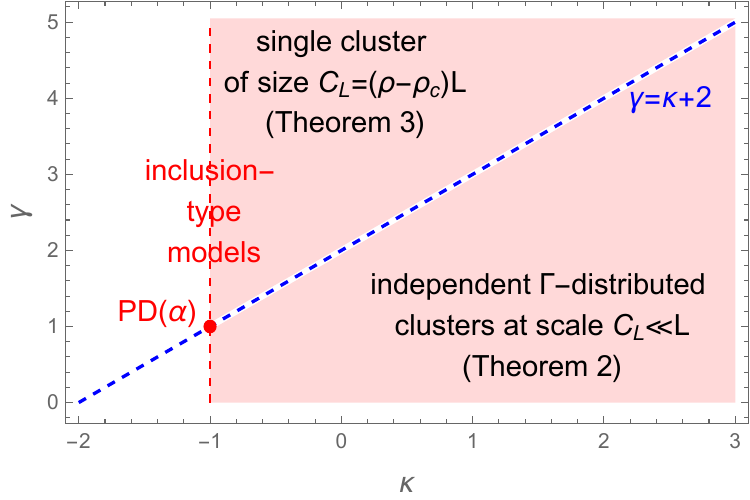}
    \caption{Phase diagram of lattice gases with stationary product weights of the form \eqref{weights}. Rigorous results of Theorem \ref{cluster} and Theorem \ref{thm3} apply in the red regions (excluding the transition line) and additional results \cite{bib_CGG22} for inclusion-type models hold for $\kappa =-1$, including a hierarchical Poisson-Dirichlet distribution PD$(\alpha )$ at the transition point (cf.\ Section \ref{sec_ips}).}
    \label{fig:phasediagram}
\end{figure}

% \medskip
\subsection{Examples of lattice gases\label{sec_ips}}

Natural examples of lattice gases on the state space $E_L =\mathbb{N}_0^\Lambda$ of particle configurations have generators of the form\vspace*{-2mm}
\begin{equation}
	\begin{aligned}
		\mathcal{L}f(\eta) = \sum_{x,y \in \Lambda} q(x,y) c(\eta_x, \eta_y) \left( f(\eta^{x y}) - f(\eta)\right) \, .
		\label{eq_generator}
	\end{aligned}\vspace*{-2mm}
\end{equation}
Here $\eta \in E_L$, $f \in C_b(E_L )$ is a bounded continuous test function and $\eta^{x y}_z = \eta_z -\delta_{z,x}+\delta_{z,y}$ is the configuration after a particle has moved from site $x$ to $y$. The kernel $q(x,y)\geq 0$ describes the spatial connections on $\Lambda$ and the kernel $c(n,m)\geq 0$ the interactions between particles. We assume in fact $c(n,m)>0$ for all $n\geq 1,m\geq 0$ and that $q$ is irreducible on $\Lambda$, so that the number of particles $N$ is the only conserved quantity with unique canonical stationary measures $\pi_{L,N}$. It is well known \cite{bib_C85,Saada2016} that for spatial homogeneity with $\sum_{y\in\Lambda} q(x,y)=\sum_{y\in\Lambda} q(y,x)$ for all $x\in\Lambda$, $\pi_{L,N}$ is a conditional product measure if and only if the kernel $c(n,m)$ satisfies the following:
	\begin{itemize}
		\item it has zero range interaction $c(n,m) = g(n)$ \hspace{3pt} \textbf{or}
		\item a curl free condition holds
		\begin{equation}
			\frac{c(n,m-1)}{c(m,n-1)} = \frac{c(n, 0)}{c(1, n-1)}\frac{c(1,m-1)}{c(m,0)} \qquad \text{for all } n,m \geq 1\ ,
			\label{eq_curl_free}\vspace*{-2mm}
		\end{equation}
		\textbf{and} the process is either 
		\begin{itemize}
			\item reversible with $q(x,y) = q(y, x)$ \hspace{3pt} \textbf{or}
			\item %spatially homogeneous \eqref{eq_irreducibility_function} and \\ 
            $c(n,m) - c(m,n) = c(n,0) - c(m,0)$\quad holds for all $n,m\geq 0$\ .
		\end{itemize}
	\end{itemize}
In this case the stationary weights are given by
\begin{equation}\label{statw}
    w(0)=1\ ,\quad w(n) = \prod^n_{m=1} \frac{c(1, m-1)}{c(m,0)}\ ,\quad n\geq 1\ .
\end{equation}

\begin{figure}
	\centering   \includegraphics[width=0.7\linewidth]{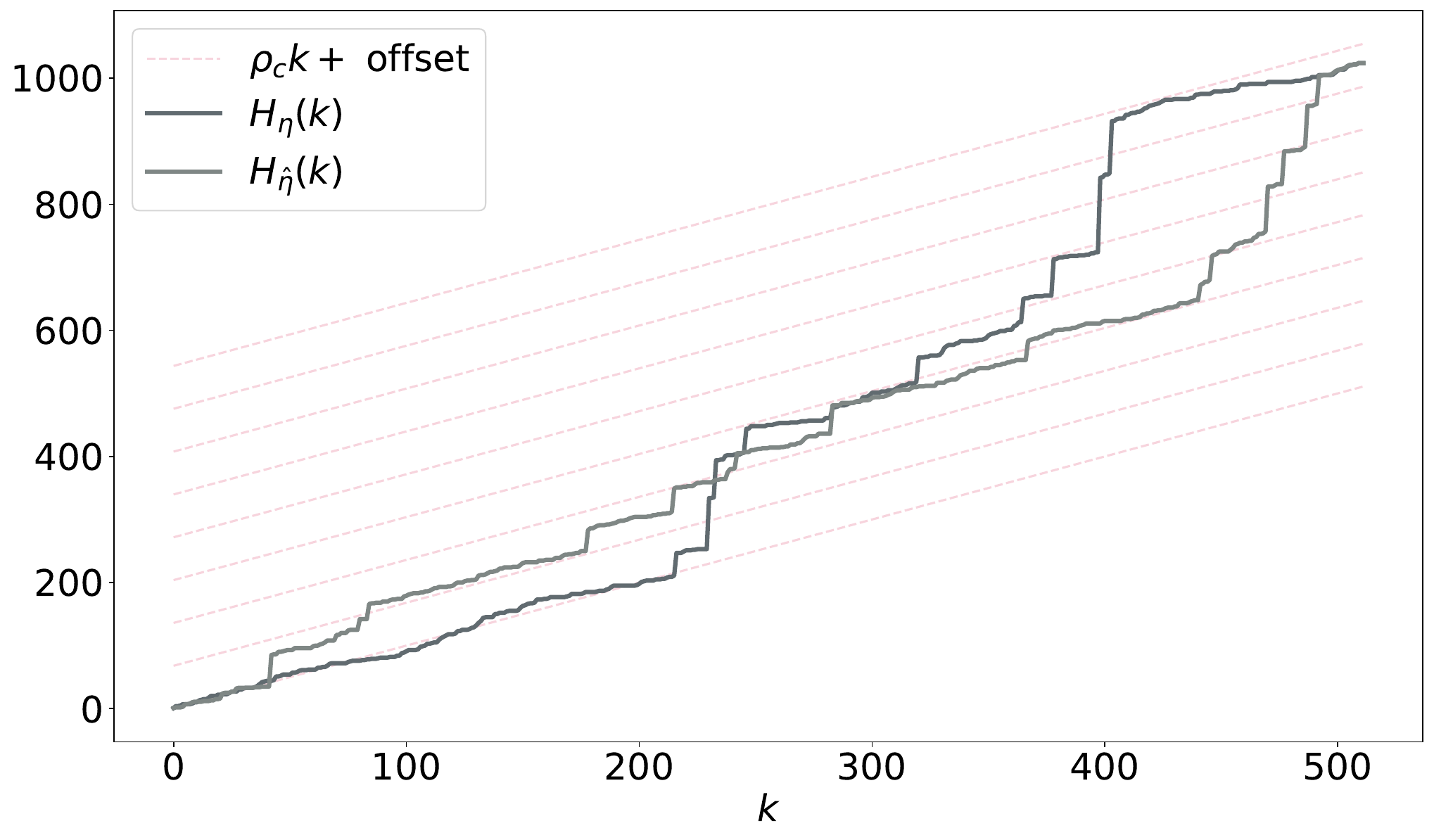}
	\caption{Accumulated density profiles $H_{\eta}(k) = \sum_{x=1}^k \eta_{x}$ on $\Lambda = \{1, \dots\ L\}$ for two independent realizations $\eta ,\eta'$ of $\pi_{L,N}$ with weights $w_L$ given in \eqref{weights} and parameters $L = 512$, $N=1024$, $\kappa = 0.5$, $\gamma =1$ and $\theta = 0.1$. Bulk weights are $w(n)=2^{-(n+1)}$ and typical profiles show various clusters with a background at density $\rho_c =1$ \eqref{wmoments}.
	}
	\label{fig:cumsum}
\end{figure}
% \fig{../img/cumsum.pdf}

A natural example of such processes with size-dependent kernel (and therefore weights) is the inclusion process (see \cite{bib_JCG20} and reference therein) where
\[
c(n,m)=n(d_L +m)\quad\mbox{with a size-dependent parameter }d_L >0\ ,
\]
that corresponds to a (small) mutation rate in applications in population genetics. This system is not quite in the class of models considered in this paper, but its special structure allows for an explicit formula $Z_{L,N} =\frac{\Gamma (N+d_L L)}{N!\Gamma (d_L L)}$ of the partition function enabling the same analysis as done here more directly. It exhibits condensation with $\rho_c =0$ for $d_L \to 0$, and with $d_L L\to \infty$ size-biased clusters on scale $C_L =1/d_L \ll L$ have an exponential distribution corresponding to the case $\kappa =-1$ in our setting. For $d_L L\to 0$ the model exhibits a single large cluster of size $\rho L$ and for $d_L L\to\alpha >0$ several clusters have asymptotically a Poisson-Dirichlet distribution PD$(\alpha )$ on the macroscopic scale $C_L =L$ as is illustrated in Figure \ref{fig:phasediagram}. For details see \cite{bib_JCG20} and \cite{bib_CGG22}, where the Poisson-Dirichlet case has been generalized to a larger class of ``inclusion-like" models. For the inclusion process also the dynamics of cluster formation is well understood \cite{grosskinsky2013dynamics,chleboun2024}, which has recently also been extended to related models with a non-trivial bulk and $\rho_c >0$ \cite{gabriel2025modulatedpoissondirichletdiffusionsarising}.

In \cite{JG2024ZRPfastrate} a zero-range process with rates
\[
c(n,m)=g(n)=\begin{cases}
    1&,\ n=1,\ldots ,A-1\,,\\
    \theta L^\gamma &,\ n=A\,,\\
    1&,\ n\geq A+1\,,
\end{cases}
\]
that are diverging as $L^\gamma$ for a particular occupation number $A\geq 1$ has been shown to exhibit condensation with a cluster scale $C_L =L^{\gamma /2}\ll L$ for $\gamma\in (0,2)$ and a $\Gamma (2,1)$-distribution of cluster sizes. This model fits exactly in the framework of this paper with $\kappa =0$, and for $\gamma >2$ exhibits a single large cluster with mass fraction $\rho -\rho_c$. The transition at $\gamma =2$ is also believed to show a hierarchical structure similar to the Poisson-Dirichlet statistics in the inclusion process. But as in the more general approach covered in this paper, this transition is hard to analyse and has only been established for the inclusion-type models so far.

\begin{figure}
	\begin{center}\mbox{\includegraphics[width=0.48\textwidth]{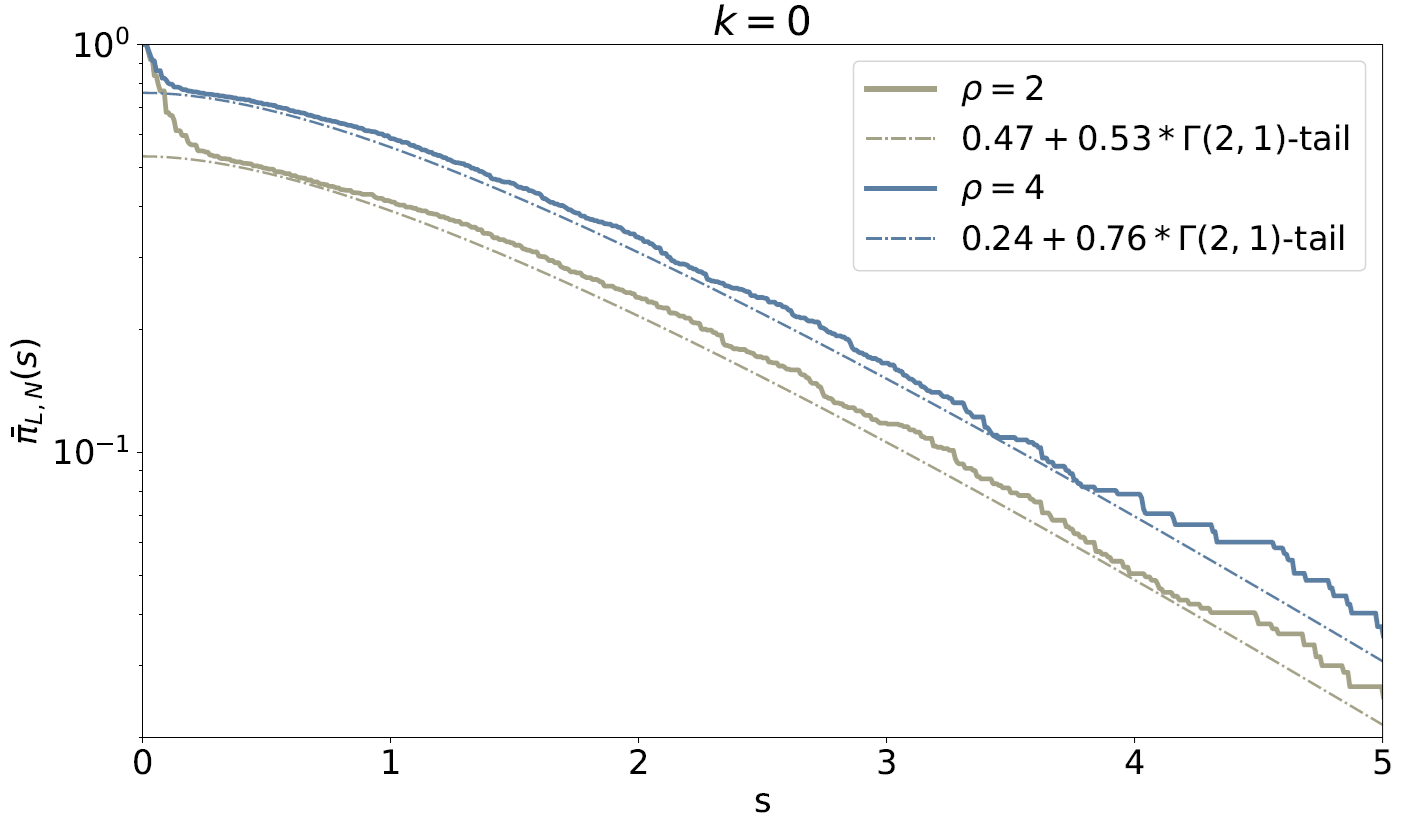}\quad\includegraphics[width=0.48\textwidth]{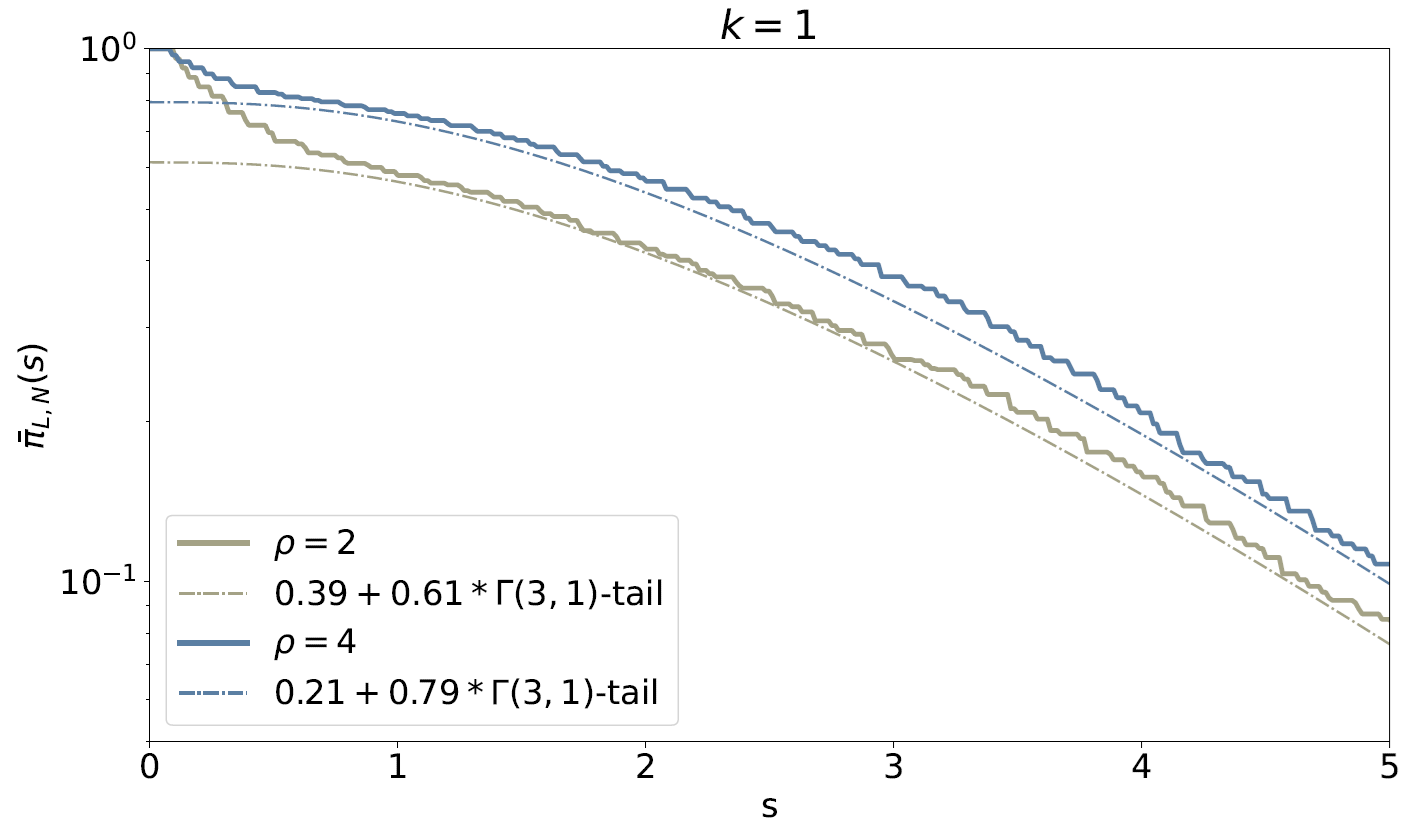}}
	\end{center}
	\caption{\label{fig_clusters}
	Size-biased empirical tail distribution $\bar{\pi}_{L,N}(s)=\frac{1}{N}\, \sum_{x \in \Lambda} \eta_x \mathbbm{1}\{\eta_x > C_L s\}$ on scale $C_L$ averaged over $48$ realizations of $\pi_{L,N}$ with weights $w_L$ given in \eqref{weights}, parameters $L = 512$, $\gamma =1$ and $\theta = 0.1$ and bulk weights $w(n)=2^{-(n+1)}$. The tails illustrate the result in Theorem \ref{cluster} and are compared to predicted asymptotic Gamma distributions, corrected for finite-size effects by using $R_L (1)$ \eqref{rl} truncated at $n\leq\rho L$ instead of $\rho_c =1$ \eqref{weights}. Note that finite-size effects are significantly larger for $\kappa =1$ with cluster scale $C_L \propto L^{1/3}$ than for $\kappa =0$ with $C_L \propto L^{1/2}$.
 %    (Right) The mean scaled tail of the ZRP simulations with the same $\rho$ and the predicted asymptotic distribution using the mean of all $\rho_{c, L}$.
	% (Yellow Lines) Simulations where $\rho = 2$. (Blue Lines) Simulations where $\rho = 4$.
}
\end{figure}

\begin{figure}
\begin{center}
	\includegraphics[width=0.7\textwidth]{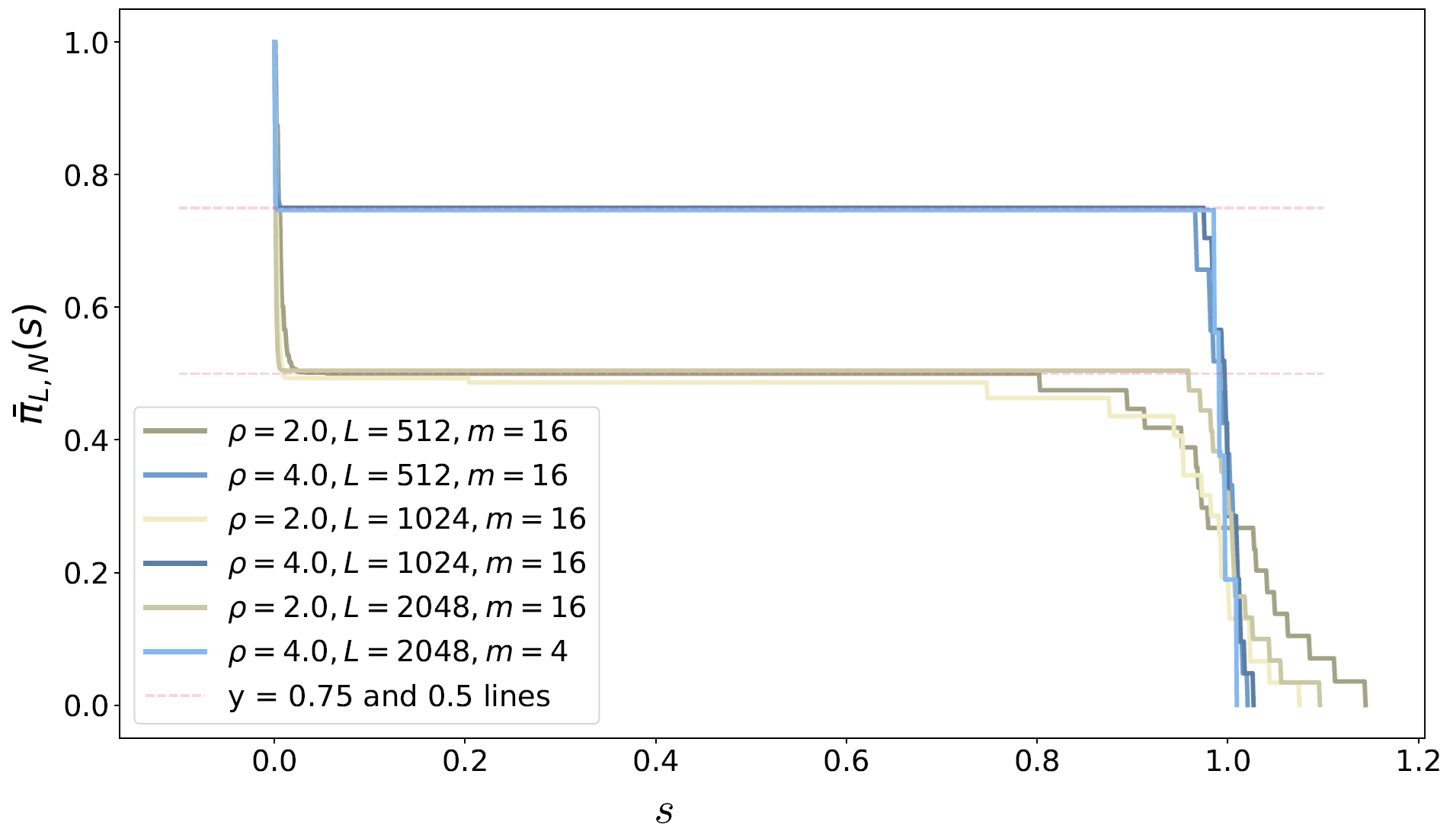}
\end{center}
\caption{\label{fig_conjecture}
	Size-biased empirical tail distribution $\bar{\pi}_{L,N}(s)=\frac{1}{N}\, \sum_{x \in \Lambda} \eta_x \mathbbm{1}\{\eta_x > C_L s\}$ on scale $C_L =(\rho -\rho_c )L$ averaged over $16$ realizations of $\pi_{L,N}$ each with weights $w_L$ given in \eqref{weights}, parameters $L = 512$, $\kappa =-1$, $\gamma =2$ and $\theta = 0.1$ and bulk weights $w(n)=2^{-(n+1)}$. The tails illustrate the result in Theorem \ref{thm3} and are compared to the predicted asymptotic step function for the tail. The fraction of mass in the condensate is $(\rho -\rho_c )/\rho =0.75$ and $0.5$, respectively, denoted by dashed red lines.
}
\end{figure}

Due to \eqref{statw}, for any given weight function $w(n)>0$ one can define a zero-range process with rates
\[
g(n)=g(1)\frac{w(n-1)}{w(n)}\ ,\quad n\geq 1
\]
that has stationary stationary product weights equal to $w$. Therefore we can directly sample from the canonical stationary measures $\pi_{L,N}$ \eqref{canon} with weights as in \eqref{weights} using a corresponding zero-range process. As opposed to the above examples, the dynamics of these models is not generic or interesting in itself, we simply use them as a means for sampling from $\pi_{L,N}$. 
Typical realizations are illustrated in Figure \ref{fig:cumsum} and in Figures \ref{fig_clusters} and \ref{fig_conjecture} we show empirical tail distributions from several realizations to illustrate Theorem \ref{cluster} and Theorem \ref{thm3}.
% We see an example of this effect with the ZRP in Figure \ref{fig:cumsum} below. 

\section{Proofs of the main results}\label{sec_proof}

\subsection{Proof of Theorem \ref{cluster}}

The central part of the proof are asymptotics for the ratio of partition functions.
For sequences $a_k ,b_k$ we write $a_k \ll b_k$ if $a_k /b_k \to 0$ and $a_k \simeq b_k$ if $a_k /b_k \to 1$ as $k\to\infty$.
% We write $f(L) \ll g(L)$ if there exists $C>0$ such that $f(L) \leq C g(L)$ for all $L$, and $f(L) \simeq g(L)$ if $\lim_{L\to \infty} f(L)/g(L) =1$.
\vspace*{-3mm}
\begin{lemma}
	\label{corollary_quotients}
	For weights $w_L(n)$ as in \eqref{weights} with $\kappa +2>\gamma \wedge 1$ we get, for the ratio of partition functions of the canonical distribution \eqref{canon} in the thermodynamic limit \eqref{thermo}, with density $\rho > \rho_c$ \eqref{wmoments}, for all $n=n_L \leq N$
	\begin{equation*}
		\frac{Z_{L-1, N - n}}{Z_{L,N}} \simeq \exp(-n/C_L)\to \begin{cases}
			1&,\ n\ll C_L ,\mbox{ in particular }n=O(1)\,,\\
			e^{-s} &,\ n\simeq C_L s\,, \\
			0 &,\ n\gg C_L\,,
		\end{cases}
	\end{equation*}
%	and thus
%	\begin{equation*}
%		\begin{aligned}
%			&\frac{Z_{L-1, N - n}}{Z_{L,N}} \to 1        &&\quad \text{for} \quad n = O(1) \qquad \text{and} \\
%			&\frac{Z_{L-1, N - n}}{Z_{L,N}} \to \exp(-s) &&\quad \text{for} \quad n = O(C_L\,s)  \, .
%		\end{aligned}
%	\end{equation*}
	where $C_L = \left(\dfrac{(\rho - \rho_c) \, L^\gamma}{\theta \, \Gamma(\kappa +2)}\right)^{1/(\kappa +2)}$ is the cluster scale from Theorem \ref{cluster}.
%	scaling factor from \eqref{eq_def_C_L}.
\end{lemma}

\noindent\textbf{Proof of Theorem \ref{cluster}.} 
We split the cumulative distribution function of $\tilde{\eta}_1/C_L$ (cf.\ \eqref{sbformula}) into the main and perturbative contributions. For fixed $u > 0$ we have
\begin{align*}
\pi_{L,N}\left[\tilde{\eta}_1 \leq C_L u\right]  =&\underbrace{\sum_{n \leq C_L u} \frac{L}{N} n w(n) \frac{Z_{L-1,N-n}}{Z_{L,N}}}_{:=S_1} +\underbrace{\sum_{n \leq C_L u} \frac{L}{N} \frac{\theta}{L^\gamma} n(n+1)^\kappa \frac{Z_{L-1,N-n}}{Z_{L,N}}}_{:=S_2}\\
&+\underbrace{\sum_{n \leq C_L u} \frac{L}{N} \frac{\theta}{L^\gamma} n(n+1)^\kappa \delta_L (n) \frac{Z_{L-1,N-n}}{Z_{L,N}}}_{:=S_3} \ .
\end{align*}
% and the polynomial part and estimate them separately.
%\begin{equation*}
%	\begin{aligned}
%		\pi_{L,N}\left[\tilde{\eta}_1 \leq C_L u\right] 
%		&= \sum_{n \leq C_L u} \frac{L}{N} n w_L (n) \frac{Z_{L-1,N-n}}{Z_{L,N}}\\
%		&= \sum_{n \leq C_L u} \frac{L}{N} n w(n) \frac{Z_{L-1,N-n}}{Z_{L,N}}\\
%		&\qquad + \sum_{n \leq C_L u} \frac{L}{N} \frac{\theta}{L} n(n+1)^k \frac{Z_{L-1,N-n}}{Z_{L,N}} \\
%		&=: S_1 + S_2
%	\end{aligned}
%\end{equation*}
We split $S_1$ again into contributions from bulk and cluster sites, introducing an intermediate auxiliary scale $A_L$ such that $A_L \to \infty$ and $A_L/C_L \to 0$,
\begin{equation}\label{split}
S_1 \simeq \frac{1}{\rho} \underbrace{\sum_{n \leq A_L} n w(n) \frac{Z_{L-1,N-n}}{Z_{L,N}}}_{:=S_{1,1}} + \frac{1}{\rho} \underbrace{\sum_{A_L < n \leq C_L\,u} n w(n) \frac{Z_{L-1,N-n}}{Z_{L,N}}}_{:=S_{1,2}}\ .
\end{equation}
For the first contribution, Lemma \ref{corollary_quotients} implies pointwise convergence of $\frac{Z_{L-1,N-n}}{Z_{L,N}} \mathbbm{1}\{ n\leq A_L\}\to 1$ for all $n\leq A_L$ and that for large enough $L$ this function is bounded by $1+\epsilon$ for some $\epsilon >0$. So we can use dominated convergence, and the definition of $\rho_c$ in \eqref{wmoments}, to get
\[
S_{1,1} \to\sum_{n\geq 0} n w(n) =\rho_c \,.
\]
Similarly, Lemma \ref{corollary_quotients} also implies that the remainder term vanishes as
%\begin{equation*}
%	\begin{aligned}
%		S_1 &\simeq \frac{1}{\rho} \sum_{n \leq C_L u} n w(n) \frac{Z_{L-1,N-n}}{Z_{L,N}} \\
%		&= \frac{1}{\rho} \sum_{n \leq A_L} n w(n) \frac{Z_{L-1,N-n}}{Z_{L,N}} + \frac{1}{\rho} \sum_{A_L \leq n \leq C_L\,u} n w(n) \frac{Z_{L-1,N-n}}{Z_{L,N}} \\
%		&\simeq \frac{1}{\rho} \frac{1}{z(1)} \underbrace{\sum_{n = 0}^\infty  n w(n)}_{=z'(1)} + \frac{1}{\rho} \sum_{A_L \leq n \leq C_L\,u} n w(n) \frac{Z_{L-1,N-n}}{Z_{L,N}} \\
%		&= \frac{\rho_c}{\rho} + S_{1, 2} \, .
%	\end{aligned}
%\end{equation*}
\begin{equation*}
	\begin{aligned}
		0 \le S_{1, 2} &\le \sum_{n> A_L} n w(n) \frac{Z_{L-1,N-n}}{Z_{L,N}} \le (1+\epsilon )\sum_{n> A_L} n w(n)\to 0\ .
	\end{aligned}
\end{equation*}
% since $e^{-s}\leq 1$. 
%The last limit follows trivially from the form of $w$ as defined in \eqref{eq_w_n_exp}, making $n w(n)$ summable. 
So $S_1 \to\rho_c /\rho$. 
For $S_2$, we substitute $n = C_L s$ , and use $N/L\to \rho$, to get:
\begin{align*}
		S_2 
%		&\simeq \frac{1}{\rho} \frac{\theta}{L} \sum_{n \leq C_L u}  n(n+1)^k \frac{Z_{L-1,N-n}}{Z_{L,N}} \\
		&\simeq \frac{1}{\rho} \frac{\theta}{L^\gamma} \sum_{\substack{s \leq u \\ s \, \in \, \mathbb{N}/C_L}} (C_L)^{\kappa +1} s^{\kappa +1} \frac{Z_{L-1,N-C_L s}}{Z_{L,N}} \\
		&%\overset{(\star)}{\simeq} 
		\simeq\frac{C_L^{k+2}}{\rho} \frac{\theta}{L^\gamma}  \sum_{\substack{s \leq u \\ s \, \in \, \mathbb{N}/C_L}} \frac{1}{C_L} s^{\kappa +1} e^{-s} \to \frac{\rho-\rho_c}{\rho} \int_0^u \frac{s^{\kappa +1}e^{-s}}{\Gamma(\kappa +2)} \, ds\ ,
\end{align*}
where the second line follows from Lemma \ref{corollary_quotients}, $\frac{Z_{L-1,N-C_L s}}{Z_{L,N}} \simeq \exp\left(-s\right)$ and applying  $C_L = \left(\dfrac{(\rho - \rho_c) \, L^\gamma}{ \theta \, \Gamma(\kappa +2)}\right)^{1/(\kappa +2)}$.

Since $S_2$ converges, it follows immediately that $S_3 \leq\sup_{n\leq N} \delta_L (n) S_2\to 0$. 
The result for $\tilde{\eta}_1, \dots , \tilde{\eta}_m$ with fixed $m \in \mathbb{N}$ can then be established iteratively using \eqref{sbformula} to factorize the canonical size-biased distribution.
%apart and then use the same argument as above for each term. For details see the proof of Theorem 2 in \cite{bib_JG24} the proof of Theorem 2 where the authors do exactly that.
\hfill $\Box$

%\subsection{Proof of Lemma \ref{corollary_quotients}}
\medskip
\subsection{Preparation for the Proof of Lemma \ref{corollary_quotients}}
\label{sec_prep_and_proof}

In this section, we collect some auxiliary results that are necessary to derive the scaling of $Z_{L, N}$ in the thermodynamic limit \eqref{thermo}, which is stated in Lemma \ref{corollary_quotients}. First, note that by assumption
\[
L\sup_{n\leq N} \delta_L (n)\to 0\,,\quad\mbox{which implies}\quad \sup_{n\leq N}(1+\delta_L (n))^L \to 1
\]
in the thermodynamic limit. 
Therefore, any relevant contribution from $\delta_L (n)$ to the scaling of $Z_{L,N}$ vanishes and we will simply set $\delta_L (n)\equiv 0$ from now on, greatly simplifying the notation in the following. We then have
% first, we write out the definition of $Z_{L, N}$ and work on restructuring its terms such that we can examine them asymptotically in the next step.
\begin{equation*}
	\begin{aligned}
		Z_{L,N} 
		&= \sum_{\eta \in E_{L,N}} \prod_{x \in \Lambda} w_L(\eta_x)= \sum_{\eta \in E_{L,N}} \prod_{x \in \Lambda} \left( w(\eta_x) + \frac{\theta}{L^\gamma} (\eta_x+1)^\kappa \right) \\
		&=  \sum_{\eta \in E_{L,N}} \, \bigg[ \prod_{x \in \Lambda} w(\eta_x) + \sum_{y \in \Lambda} \frac{\theta}{L^\gamma} (\eta_y+1)^\kappa \prod_{\substack{x \in \Lambda \\ x \ne y}}w(\eta_x) \\
		&\qquad \qquad\quad +  \sum_{\substack{y_1, y_2 \in \Lambda \\ y_1 \ne y_2}} \frac{\theta^2}{L^{2\gamma}} (\eta_{y_1}+1)^\kappa (\eta_{y_2}+1)^\kappa \prod_{\substack{x \in \Lambda \\ x \ne y_1 \\ x \ne y_2}}w(\eta_x) + \dots \bigg]
	\end{aligned}
\end{equation*}
With the usual notation $w^L \Big(\sum\limits_{x=1}^L \eta_x =N\Big) =\sum\limits_{\eta \in E_{L,N}} \prod\limits_{x \in \Lambda} w(\eta_x)$ for product weights we get
\begin{align}
    Z_{L,N} 
		&= w^L \Big(\sum_{x=1}^L \eta_x =N\Big) +L\sum_{n=0}^N w^{L-1} \Big(\sum_{x=1}^{L-1} \eta_x =N-n\Big) \frac{\theta}{L^\gamma} (n+1)^\kappa +\ldots\nonumber\\
        &=w^L \Big(\sum_{x=1}^L \eta_x =N\Big) +\sum_{\ell =1}^L {L\choose\ell} \frac{\theta^\ell}{L^{\ell\gamma}} \sum_{n=0}^N w^{L-\ell} \Big(\sum_{x=1}^{L-\ell} \eta_x =N-n\Big)\sum_{\substack{n_1 ,\ldots ,n_\ell =0\\ n_1 +..+n_\ell =n}}^n \prod_{y=1}^\ell  (n_y +1)^\kappa \nonumber\\
        &=: S_0 +\sum_{\ell =1}^L {L\choose\ell} \frac{\theta^\ell}{L^{\ell\gamma}} \sum_{n=0}^N S_{\ell ,n}\label{toshow}
        % &=: b_{0,0} +\sum_{\ell = 1}^L {L\choose\ell} \frac{\theta^\ell}{L^{\ell\gamma}}\sum_{n=0}^N b_{\ell, n}\quad\mbox{with}\quad b_{\ell, n}=w^{L-\ell} \Big(\sum_{x=1}^{L-\ell} \eta_x =N-n\Big)\sum_{\substack{n_1 ,\ldots ,n_\ell =0\\ n_1 +\dots+n_\ell =n}}^n \prod_{y=1}^\ell  (n_y +1)^\kappa
\end{align}

\noindent We will first examine the asymptotics of the contribution of the perturbation.

\begin{lemma}
	\label{lemma_sumprod}
	For $\kappa > -1$, we get for $\ell\ll n$ the asymptotics
	\begin{equation*}
		A_\ell(n):=\sum_{\substack{n_1 ,\ldots ,n_\ell =0\\ n_1 +..+n_\ell =n}}^n \prod_{y=1}^\ell(n_y + 1)^\kappa 
			= n^{(\kappa +1) \, \ell - 1} \frac{\Gamma(\kappa +1)^\ell}{\Gamma((\kappa +1) \, \ell)} \big( 1+O(\ell /n)\big)\quad\mbox{as }n\to\infty\ .
	\end{equation*}
%    For $\kappa <-1$ we get\quad $\displaystyle \sum_{\substack{n_1 ,\ldots ,n_\ell =0\\ n_1 +..+n_\ell =n}}^n \prod_{y=1}^\ell(n_y + 1)^\kappa \simeq n^\kappa \ell \zeta (-\kappa )^\ell$\ ,\quad where $\displaystyle\zeta (-\kappa )=\sum_{m=1}^\infty m^\kappa <\infty$\quad is the Riemann zeta function, as long as $n\gg \ell^{-1/(\kappa +1)}$.
For $n/\ell\to u>0$, we get the asymptotic behaviour
\[
A_\ell(n) \simeq C_u \frac{\varphi (u)^{-n} z_\kappa (\varphi (u))^\ell}{\sqrt{\ell }}
% \quad\mbox{with constants }C_u >0,\ \varphi (u)\in (0,1)\ .
\]
with constants $C_u >0$, $\varphi (u)\in (0,1)$ and $z_\kappa (\phi )$ defined in \eqref{zkappa} below.\\
% \textcolor{red}{PC: I believe we also have to control small n? Maybe this is obviously?}
If $n \leq C \ell$ for some $C\geq 1$, there exists $C_0$ independent of $n$ and $\ell$ such that
\begin{equation}\label{finalbound}
   A_\ell(n) \leq C_0^\ell n^{\kappa \ell}\ .
\end{equation}
\end{lemma}

\begin{proof} 
%\noindent\textbf{Proof.} 
For $\ell\ll n$, we can do a direct asymptotic analysis. Substituting $u_y = \frac{n_y}{n}$, we get
	\begin{equation*}
		\begin{aligned}
			\sum_{n_1 + \dots + n_\ell = n}  \prod_{y=1}^\ell (n_y + 1)^\kappa 
			&= (n+1)^{\kappa \, \ell} \sum_{\substack{u_1 + \dots + u_\ell = 1 \\ u_y \in \mathbb{N}/(n+1)}} \prod_{y=1}^\ell u_y^\kappa \\
			&= (n+1)^{(\kappa +1)\,\ell-1} \sum_{\substack{u_1 + \dots + u_{\ell-1} \le 1 \\ u_y \in \mathbb{N}/(n+1)}} \frac{1}{(n+1)^{\ell - 1}} \bigg( 1-\sum_{i=1}^{\ell-1} u_y \bigg)^\kappa \prod_{y=1}^{\ell -1} u_y^\kappa \\
			&= n^{(\kappa +1)\,\ell-1} \int_{\Delta_\ell} \bigg( 1-\sum_{y = 1}^{\ell-1} u_y \bigg)^\kappa \prod_{y=1}^{\ell-1} u_y^\kappa d^d u\,\Big( 1+O(\ell /n)\Big)\\
			&\overset{(\star)}{=}  n^{(\kappa +1)\,\ell-1} \frac{\Gamma(\kappa +1)^\ell}{\Gamma((\kappa +1) \, \ell)}\,\Big( 1+O(\ell /n)\Big)\quad\mbox{as }n\to\infty ,\ \ell\ll n\ ,
		\end{aligned}
	\end{equation*}
	where $\Delta_{\ell+1} := \{u \in \mathbb{R}^{\ell} \mid u_y > 0, y = 1, \dots, \ell\mbox{ and } u_1 + \dots + u_{\ell} \le 1 \}$ and $(\star)$ follows as a Dirichlet integral from \cite{bib_M10} (equations (1) and (3)).
    
We now consider the case $n/\ell \to \mu > 0$. For general $\ell$ and $n$, first note that for all $\phi\in (0,1)$
\begin{align}
    \label{eq:tilttwo}
A_\ell(n)=\sum_{\substack{n_1 ,\ldots ,n_\ell =0\\ n_1 +..+n_\ell =n}}^n \prod_{y=1}^\ell(n_y + 1)^\kappa =\phi^{-n} z_\kappa (\phi )^\ell \nu_{\kappa ,\phi }^{\ell} \Big(\sum_{y=1}^\ell \eta_y =n\Big)\ ,
\end{align}
where $\nu_{\kappa ,\phi}^{\ell}$ is a product probability measure on $\mathbb{N}_0^\ell$ with marginal 
\begin{equation}\label{zkappa}
    \nu_{\kappa ,\phi} (n)=(n+1)^\kappa \phi^n /z_\kappa (\phi ) \quad\mbox{and normalization}\quad z_\kappa (\phi )=\sum_{n=0}^\infty (n+1)^\kappa \phi^n  \,.
\end{equation}
Therefore, we can do a standard tilting to interpret the sum as a typical probability after normalization. For $\phi \in (0,1)$ define
\[
r_\kappa (\phi )=\frac{1}{z_\kappa (\phi )} \sum_{n=0}^\infty n(n+1)^\kappa \phi^n< \infty\ ,
\]
which is strictly monotone and invertible with $r_\kappa (0)=0$ and $z_\kappa (\phi ),\, r_\kappa (\phi )\to\infty$ for $\phi\nearrow 1$. Therefore, we can choose
\[
\phi =\varphi (n/\ell )\quad\mbox{such that}\quad r_\kappa \big(\varphi (n/\ell )\big) =n/\ell\quad\mbox{for all }n\geq 0,\ \ell\geq 1\ .
\]
With this choice of $\varphi (n/\ell )$, the event $\sum_{y=1}^\ell \eta_y =n$ is typical for $\nu_{\kappa ,\phi}^{\ell}$, and therefore the local limit theorem for triangular arrays (see \cite{bib_DM95} and \cite[Proposition A.1]{bib_CGG22}) implies
\[
\nu_{\kappa ,\varphi (n/\ell )}^{\ell} \Big(\sum_{y=1}^\ell \eta_y =n\Big) \simeq \frac{1}{\sqrt{2\pi \ell\sigma_\kappa^2 (n/\ell )}}\quad\mbox{as }\ell\to\infty\ ,
\]
with variance $\sigma_\kappa^2 (u) =\langle\nu_{\kappa ,\varphi (u)} ,\eta_x^2 \rangle -u^2$. So for $n/\ell \to u>0$ the leading order asymptotics is given by
\[
\sum_{\substack{n_1 ,\ldots ,n_\ell =0\\ n_1 +..+n_\ell =n}}^n \prod_{y=1}^\ell(n_y + 1)^\kappa \simeq\frac{\varphi (u)^{-n} z_\kappa (\varphi (u))^\ell}{\sqrt{2\pi \ell\sigma_\kappa^2 (u)}}\quad\mbox{with }\varphi (u)<1\ . 
\]
Note that this is consistent with the case $\ell /n\to 0$. Here $\varphi (n/\ell )\to 1$, and we can use the well-known asymptotics of the polylogarithm
\begin{align}
    \label{eq:polylog}
    z_\kappa (\phi )=\frac{1}{\phi}\mathrm{Li}_{-\kappa} (\phi )\simeq \frac{\Gamma (\kappa +1)}{(1-\phi )^{\kappa +1}}\quad \mbox{as }\phi\nearrow 1\,.
\end{align}
The final bound \eqref{finalbound} follows immediately by bounding each term in the product by $(n+1)$ and using $n \leq C\ell$,
\[
\sum_{\substack{n_1 ,\ldots ,n_\ell =0\\ n_1 +..+n_\ell =n}}^n \prod_{y=1}^\ell(n_y + 1)^\kappa \leq (n+1)^{\kappa \ell}\binom{n+\ell-1}{\ell-1} \leq C_0^\ell n^{\kappa \ell}
\]
to derive the same result.
%\hfill $\Box$\\
\end{proof}

%    For $\kappa <-1$ the terms $(n_y +1)^\kappa$, $n_y =0,1,\ldots$ are summable with
%    \[
%    \sum_{m=1}^\infty m^\kappa =\zeta (-\kappa )<\infty \quad\mbox{given by the Riemann zeta function}\ .
%    \]
%    Then we can write 
%    \[
%    \sum_{\substack{n_1 ,\ldots ,n_\ell =0\\ n_1 +..+n_\ell =n}}^n \prod_{y=1}^\ell(n_y + 1)^\kappa =\zeta (-\kappa )^\ell \bar\nu^\ell \bigg[\sum_{y=1}^\ell \eta_y =n\bigg]
%    \]
%    for a product measure $\bar\nu^\ell$ with marginals $(m+1)^\kappa /\zeta (-\kappa )$, $m\geq 0$. These have polynomial tail and the classical local limit theorem for iid random variables with power-law tails \cite{doney2001local,nagaev1979large} implies that
%    \[
%    \bar\nu^\ell \bigg[\sum_{y=1}^\ell \eta_y =n\bigg] \simeq \ell n^\kappa \quad\mbox{as } 1\ll\ell \ll n\ ,
%    \]
%    which implies the result for $1\ll \ell\ll L$. For $\ell =O(1)$ we can use the result in \cite{ferrari2007condensation} for a fixed number of random variables with power-law tail to get the same scaling as above.\hfill $\Box$\\

% The following lemma is like Lemma \ref{lemma_Ll}, mainly an exercise in the asymptotic behavior of binomial coefficients; additionally, we prepare for using Laplace's method in Lemma \ref{theorem_z_l_n_asymp}.

Recall that with \eqref{tilt} for all $\rho\geq 0$ there exists $\phi =\Phi (\rho )\in [0,\bar\phi )$ such that $R(\phi )=\rho$. Using a standard local limit theorem \cite{bib_DM95, borovkov2020} for the tilted measure $\nu_\phi (n):=w(n)\phi^n /z(\phi )$ this implies the following large deviation result.

\begin{lemma}\label{lemma_binomexp}
For $\ell\ll L$, $N/L\to\rho$ and $n/L\to u<\rho$ we get
    \[
    w^{L-\ell} \Big(\sum_{x=1}^{L-\ell} \eta_x =N-n\Big) \simeq \frac{1}{\sqrt{2\pi (L-\ell )\sigma^2 (\rho -u)}} e^{-(L-\ell )I(\rho -u)}    
    % z\big(\Phi (\rho -u)\big)^{L-\ell} \Phi (\rho -u)^{n-N}
        % \sqrt{\frac{1+\rho-u}{2 \pi \, (\rho-u)}} 
    \]
    where $\sigma^2 (\rho -u)=\langle \nu_{\Phi (\rho -u)} ,\eta_x^2 \rangle -\langle\nu_{\Phi (\rho -u)} ,\eta_x \rangle^2$ is the variance of the tilted measure and
    \[
        I(u)=\sup_{\phi\in [0,\phi_c )} \big( u\ln\phi -\ln z(\phi )\big) =u\ln\Phi (u)-\ln z(\Phi (u))
    \]
    is the usual rate function given as the Legendre transform of the cumulant generating function $\ln z(\phi )$. $I(u)\geq 0$ and $I(u)=0$ if and only if $u=\rho_c =\sum_{n\geq 0} nw(n)$.
    % and $\phi$ is the standard normal density function.
	% For $n \propto L, \ell \propto L^\alpha, \alpha \in [0, 1)$ and $\rho > \rho_c$
	% \begin{equation*}
	% 	\begin{aligned}
	% 		&\binom{N-n+L-\ell-1}{N-n} D^{L-\ell} \exp(-C(N - n)) \\
	% 		&\qquad\simeq \frac{1}{\sqrt{L}} \: \sqrt{\frac{1+\rho-\frac{n}{L}}{2 \pi \, (\rho-\frac{n}{L})}}\exp\left((L-\ell)\, h \left(\frac{n}{L} \right) \right)\, ,
	% 	\end{aligned}
	% \end{equation*}
	% where 
	% \begin{equation}
	% 	\begin{aligned}
	% 		\label{eq_def_h}
	% 		h(u) := (1&+\rho-u) \ln(1+\rho-u)  \\
	% 		&- (\rho-u) \left( \ln(\rho-u) + C \right) + \ln(D) \, . 
	% 	\end{aligned}
	% \end{equation}
	% The maximum of $h(u)$ is 0 at $u = \rho - \rho_c$ and $h''(\rho - \rho_c) = 2 - e^{-C} - e^C<0$ for $C \ne 0$. The constants $C$ and $D$ are from the weights \eqref{lemma_expweights}.
\end{lemma}

%\noindent\textbf{Proof.} 
\begin{proof}
    See, for example, Theorem 2.2.5 in \cite{borovkov2020}.
\end{proof}
\begin{remark}
\label{rem:lemma_binomexp}
    % \textcolor{red}{PC: Possibly add remark to Lemma above?}
    The case $n/L \to \rho$ is not covered by Lemma \ref{lemma_binomexp}. However, we get an exponentially small upper bound when $N-n$ is small by a standard Chernoff argument, as follows. Since the weights, $w$, are normalised and $w(0) < 1$,
    \[
    z(\phi) = \sum_{n=0}^\infty w(n)\phi^n \leq 1 - \frac{1}{2}w(\eta_1 > 0) < 1 \quad \textrm{for } \phi \in (0,1/2]\,.
    \]
    It follows that there exists $c>0$ independent of $m$ such that for all $\epsilon$ sufficiently small (depending on $c$)
    \[
     w^{M} \Big(\sum_{x=1}^{M} \eta_x \leq \epsilon M\Big) \leq z(1/2)^M 2^{-\epsilon M} = e^{M(\epsilon\log 2 + \log z(1/2))} \leq e^{-cM}\,.
    \]
    This is sufficient to control the regime when $N-n$ is small.
\end{remark}

\subsection{Scale of $Z_{L,N}$ and Proof of Lemma \ref{corollary_quotients}}
%%%%%%%% MIM's $%%%%%%%%%%

Using Lemma \ref{lemma_binomexp} with $\ell =0$ and $u=0$, the first term in \eqref{toshow}
\begin{equation}\label{zerorder}
S_0 =w^L \Big(\sum_{x=1}^L \eta_x =N\Big) \simeq \frac{1}{\sqrt{2\pi L\sigma^2 (\rho )}} e^{-L\, I(\rho)} 
\end{equation}
is exponentially small in $L$ since $I(\rho )>0$ for $\rho >\rho_c$.

To control the dominant contribution to the $\ell$ sum in \eqref{toshow} we introduce a cut-off at a scale $\tilde{L}$ such that $0 \ll C_L \ll \tilde{L} \ll L$. For $\ell\leq \tilde L:=L^{1-\gamma /(2(\kappa +2))} \ll L$ and $n \gg \ell$, we can apply Lemma \ref{lemma_sumprod} and again Lemma \ref{lemma_binomexp} to get
\begin{align*}
    S_{\ell ,n} &=w^{L-\ell} \Big(\sum_{x=1}^{L-\ell} \eta_x =N-n\Big)\sum_{\substack{n_1 ,\ldots ,n_\ell =0\\ n_1 +..+n_\ell =n}}^n \prod_{y=1}^\ell  (n_y +1)^k \\
    &\simeq \frac{L^{(\kappa +1)\,\ell-1}}{\sqrt{2\pi (L-\ell )\sigma^2 (\rho -u)}} e^{-(L-\ell )I(\rho -u)} u^{(\kappa +1)\,\ell-1} \frac{\Gamma(\kappa +1)^\ell}{\Gamma((\kappa +1) \, \ell)}
\end{align*}
in the thermodynamic limit when $n/L\to u\in [0,\rho)$. For $n \gg \ell$ and $n/L\to \rho$ then Lemma \ref{lemma_binomexp} no longer applies. In this case, we can apply Remark \ref{rem:lemma_binomexp} and again by Lemma \ref{lemma_sumprod} we have $S_{\ell,n}\simeq e^{-c(L-\ell)}n^{(\kappa +1) \, \ell - 1} \frac{\Gamma(\kappa +1)^\ell}{\Gamma((\kappa +1) \, \ell)}$. On the other hand, if $n \leq C \ell$, Lemma \ref{lemma_binomexp} applies, and we can use \eqref{finalbound} in Lemma \ref{lemma_sumprod} to observe $S_{\ell,n} \simeq e^{-(L-\ell)I(\rho)}e^{k\ell\log \ell}$. 
% \textcolor{red}{PC: Maybe we comment that we leave the details out, they seem simple to check though I only did it roughly on tablet. Both these terms are significantly lower order than the saddle point.}
The properties of the rate function $I$ in Lemma \ref{lemma_binomexp} imply that summation over $n$ is then dominated by the saddle point at $u=\rho -\rho_c$, and we get asymptotically
\begin{align*}
    \sum_{n=0}^N S_{\ell ,n}&\simeq \frac{L^{(\kappa +1)\,\ell-1}}{\sqrt{2\pi (L-\ell )}} \frac{\Gamma(\kappa +1)^\ell}{\Gamma((\kappa +1) \, \ell)}L\int_0^\rho e^{-(L-\ell )I(\rho -u)} \frac{u^{(\kappa +1)\,\ell-1}}{\sqrt{\sigma^2 (\rho -u)}}\, du\\
    &\simeq \frac{L^{(\kappa +1)\,\ell-1}}{\sigma^2} \frac{\Gamma(\kappa +1)^\ell}{\Gamma((\kappa +1)\ell )}\, (\rho -\rho_c )^{(\kappa +1)\,\ell-1}
\end{align*}
with Laplace method using $I''(\rho_c )=\sigma^2$. Plugging this into \eqref{toshow} and using the standard Stirling formula with $\ell\ll L$
\[
\ln {L\choose\ell} \simeq \ell \Big( 1+\ln L-\ln\ell -\ell /L\Big)\ ,
\]
we get for the summation over $\ell$ \begin{align*}
    &\sum_{\ell =1}^{\tilde L} {L\choose\ell} \frac{\theta^\ell}{L^{\ell\gamma}} \frac{L^{(\kappa +1)\,\ell-1}}{\sigma^2} \frac{\Gamma(\kappa +1)^\ell}{\Gamma((\kappa +1)\ell )}\, (\rho -\rho_c )^{(\kappa +1)\,\ell-1} \simeq\frac{(\rho -\rho_c )^{-1}}{L\sigma^2}\\
    &\quad\times\sum_{\ell =1}^{\tilde L} \exp\bigg(\ell\Big( \ln\frac{L}{\ell} {-}\frac{\ell}{L} {-}\gamma\ln L{+}(\kappa {+}1)\ln \frac{L(\rho {-}\rho_c)}{(\kappa {+}1)\ell}+\ln (\theta\Gamma (\kappa {+}1))+(\kappa {+}2)\Big)\bigg)\ .
\end{align*}
Substituting $\ell =sL^\alpha$ for some $\alpha\in [0,1]$, we see that the leading order contribution to the exponent is
\[
sL^\alpha \ln L\Big( (1-\alpha )-\gamma +(\kappa +1)(1-\alpha )\Big)\ .
\]
To make the prefactor non-negative and to maximize the scale of the exponent this implies the choice $\alpha = 1-\gamma /(\kappa +2)$, which also implies $\ell =sL^\alpha \ll\tilde L$. With this scaling, the above sum is asymptotically equal to
\[
    \sum_{\ell =1}^{\tilde L} {L\choose\ell} \frac{\theta^\ell}{L^{\ell\gamma}} \sum_{n=0}^N S_{\ell ,n} 
    \simeq\frac{L^{\alpha -1}}{\sigma^2 (\rho -\rho_c )} \int_0^\infty e^{L^\alpha (f(s)+o(1))}\, ds\ 
\]
where $f(s) =s\Big( (\kappa +2)(1-\ln s)+(\kappa+1)\ln\big(\frac{\rho -\rho_c}{\kappa +1}\big)+\ln \big(\theta\Gamma (\kappa +1)\big) \Big)\ $.
    % \exp\ell\Big( 1+\ln\frac{L}{\ell} -\frac{\ell}{L}+\ln\theta -\gamma\ln L+(\kappa +1)\ln N+\ln \Gamma (\kappa +1)\Big) ds

Applying the standard Laplace method, the integral is dominated by the saddle point $s^*$ with $f'(s^* )=0$ and we get
\[    
\int_0^\infty e^{L^\alpha (f(s)+o(1))}\, ds\simeq \sqrt{\frac{2\pi s^* (\rho )}{L^\alpha (\kappa +2)}} e^{L^\alpha ((\kappa +2) s^* (\rho ) +o(1))}\ ,
\]
where 
$s^* (\rho ) =(\rho -\rho_c)^{\frac{\kappa +1}{\kappa +2}} C_{\kappa ,\theta}$ with $C_{\kappa ,\theta} =\frac{(\theta\Gamma (\kappa +2))^{1/(\kappa +2)}}{\kappa +1}$.

The dominating scale is $\ell=sL^\alpha \ll\tilde L=L^{1-\gamma /(2\kappa +4)}$. 
To show that the terms with $\ell\geq\tilde L$ do not contribute to the asymptotic of $Z_{L,N}$, observe
\[
\sum_{n=0}^N S_{\ell ,n} \leq \sum_{n=0}^N w^{L-\ell} \Big(\sum_{x=1}^{L-\ell} \eta_x =N-n\Big)A_{\ell}(n) \leq  \sum_{n=0}^N A_{\ell}(n)\,,
\]
since $w^{L-\ell}$ is a probability measure.
Now, using \eqref{eq:tilttwo}, for $\phi \in (0,1)$,
\[
\sum_{n=0}^N A_{\ell}(n) \leq \phi^{-N}z_\kappa (\phi)^\ell\,.
\]
Choosing $\phi = \exp\left(-\frac{\ell}{L\log L}\right)\in (0,1)$, and applying \eqref{eq:polylog} to bound $z_\kappa (\phi) \leq C_\kappa (1-\phi)^{-(\kappa +1)}$ for $L$ sufficiently large, this gives
\[
\sum_{n=0}^N A_{\ell}(n) \leq C_\kappa^\ell \left( \frac{L \log L}{\ell}\right)^{(\kappa+1)\ell} \exp\left(\frac{\rho \ell}{\log L}\right)\,.
\]
Using the standard upper bound $\binom{L}{\ell}\leq \frac{e^{\ell}L^{\ell}}{\ell^\ell}$, then for $L$ sufficiently large there exists $\epsilon < \gamma/4$ such that
\[
\sum_{\ell \geq \tilde{L}}^L {L\choose\ell} \frac{\theta^\ell}{L^{\ell\gamma}} \sum_{n=0}^N S_{\ell ,n} \leq \sum_{\ell \geq \tilde{L}}^LC_{\kappa,\theta}^\ell(\log L)^{\ell}\left(\frac{L^{\kappa+2-\gamma}}{\ell^{\kappa+2}}\right)^\ell\leq \sum_{\ell \geq \tilde{L}}^L L^{(-\frac{\gamma}{2}+\epsilon)\ell}\,,
\]
where the final bound follows from $\ell^{\kappa+2} \geq \tilde{L}^{\kappa+2}=L^{\kappa+2 - \frac{\gamma}{2}}$.
Therefore the sum in \eqref{toshow}, $\displaystyle Z_{L,N} =S_0 +\sum_{\ell =1}^L {L\choose\ell} \frac{\theta^\ell}{L^{\ell\gamma}} \sum_{n=0}^N S_{\ell ,n}$, is dominated by $\ell\leq\tilde L$. 
Together with \eqref{zerorder}, this implies the asymptotics
\[
Z_{L,N} \simeq \sqrt{\frac{2\pi s^* (\rho )}{\kappa +2}}\frac{L^{\alpha -1}}{\sigma^2 (\rho -\rho_c )}e^{L^\alpha ((\kappa +2) (\rho -\rho_c)^{\frac{\kappa +1}{\kappa +2}} C_{\kappa ,\theta} +o(1))}\ .
\]
To prove Lemma \ref{corollary_quotients}, we are interested in the leading order behaviour of the ratio $Z_{L-1,N-n}/Z_{L,N}$, which only depends on the leading order contribution in the exponent. With $\frac{N-n}{L-1} \simeq\rho -(n-\rho )/L$ this leads to
\[
\frac{Z_{L-1,N-n}}{Z_{L,N}} \simeq \exp\bigg( -L^{\alpha -1} (n-\rho ) \Big(\frac{\theta\Gamma (\kappa +2)}{\rho -\rho_c}\Big)^{1/(\kappa +2)} \bigg)\ ,
\]
which has the desired asymptotics since $\alpha -1=-\frac{\gamma}{\kappa +2}$. In particular, for $n\simeq C_L s$ with $C_L =\left(\dfrac{(\rho - \rho_c) \, L^\gamma}{\theta \, \Gamma(\kappa +2)}\right)^{1/(\kappa +2)}$, we have $Z_{L-1,N-n}/Z_{L,N}\to e^{-s}$.\\

\subsection{Proof of Theorem \ref{thm3}}

\begin{lemma}
\label{zasymp3}
For weights $w_L(n)$ as in \eqref{weights} with $\gamma >\kappa +2> 1$ we get with $N/L\to\rho >\rho_c$
\[
Z_{L,N}\simeq \theta L^{\kappa +1 -\gamma} (\rho -\rho_c )^\kappa\ .
\]
\end{lemma}

\begin{proof}
%\noindent\textbf{Proof.} 
Using the decomposition \eqref{toshow}
\[
Z_{L,N} =S_0 +\sum_{\ell =1}^L {L\choose\ell} \frac{\theta^\ell}{L^{\ell\gamma}} \sum_{n=0}^N S_{\ell ,n}
\]
we know from the proof of Lemma \ref{corollary_quotients} that $S_0$ is exponentially small in $L$ \eqref{zerorder} and we arrive complete analogously at
\[
Z_{L,N} \simeq \sum_{\ell =1}^{L} {L\choose\ell} \frac{\theta^\ell}{L^{\ell\gamma}} \frac{L^{(\kappa +1)\,\ell-1}}{\sigma^2} \frac{\Gamma(\kappa +1)^\ell}{\Gamma((\kappa +1)\ell )}\, (\rho -\rho_c )^{(\kappa +1)\,\ell-1}\ .
\]
For $\gamma >\kappa +2$, the asymptotics of this sum is dominated by the first term $\theta L^{\kappa +1-\gamma} (\rho-\rho_c )^\kappa$. With ${L\choose\ell }\leq L^\ell/ \ell !$ and $\frac{\Gamma(\kappa +1)^\ell}{\Gamma((\kappa +1)\ell )}\leq 1$ for $\kappa \geq 0$ ($\kappa\in (-1,0)$ can be treated similarly with an additional constant) we have the upper bound
\begin{align*}
    Z_{L,N} &\leq\frac{1}{L(\rho -\rho_c )}\sum_{\ell =1}^\infty \frac{1}{\ell !} \Big(\theta L^{\kappa +2-\gamma} (\rho -\rho_c)^{\kappa +1}\Big)^\ell =\frac{e^{\theta L^{\kappa +2-\gamma} (\rho -\rho_c)^{\kappa +1}}-1}{L(\rho -\rho_c)}\\
    &\simeq \theta L^{\kappa +1-\gamma} (\rho -\rho_c)^\kappa\quad\mbox{as }L\to\infty\ ,
\end{align*}
which matches the first term of the sum.
%\hfill $\Box$\\
\end{proof}

% \begin{lemma}
% 	\label{lemmathm3}
% 	For weights $w_L(n)$ as in \eqref{weights} with $\gamma >\kappa +2> 1$ we get for the ratio of partition functions of the canonical distribution \eqref{canon} in the thermodynamic limit \eqref{thermo} with density $\rho > \rho_c$ \eqref{wmoments} that
% 	\begin{equation*}
% 		\frac{Z_{L-1, N - n}}{Z_{L,N}} \to \begin{cases}
% 			(\frac{\rho -\rho_c -u}{\rho -\rho_c})^\kappa &,\ n/L\to u<\rho -\rho_c \\
% 			0 &,\ n/L\to u\geq \rho -\rho_c
% 		\end{cases}\ .
% 	\end{equation*}
% \end{lemma}

\noindent\textbf{Proof of Theorem \ref{thm3}.} With the above Lemma we can prove Theorem \ref{thm3} in full analogy to Theorem \ref{cluster}. 
We again split the cumulative distribution function of $\tilde{\eta}_1/L$ (cf.\ \eqref{sbformula}) into the main and the perturbative contributions with $u\in [0,\rho ]$
\begin{align}
	\pi_{L,N}\left[\tilde{\eta}_1 \leq L u\right]  =&\underbrace{\sum_{n \leq L u} \frac{L}{N} n w(n) \frac{Z_{L-1,N-n}}{Z_{L,N}}}_{:=S_1} +\underbrace{\sum_{n \leq L u} \frac{L}{N} \frac{\theta}{L^\gamma} n(n+1)^\kappa \frac{Z_{L-1,N-n}}{Z_{L,N}}}_{:=S_2}\nonumber\\
	&+\underbrace{\sum_{n \leq L u} \frac{L}{N} \frac{\theta}{L^\gamma} n(n+1)^\kappa \delta_L (n) \frac{Z_{L-1,N-n}}{Z_{L,N}}}_{:=S_3}\ .\label{toshow3}
\end{align}
For the first term $S_1$, we get for all $0<u<\rho -\rho_c$ that $\frac{Z_{L-1,N-n}}{Z_{L,N}}\to 1$ and thus by bounded convergence
\[
S_1 \to \frac{1}{\rho}\sum_{n \leq L u} \ n w(n)=\frac{\rho_c}{\rho}\quad\mbox{as }L\to\infty .
\]
The same holds for $u\in [\rho -\rho_c ,\rho ]$ as will become clear from the behaviour of $S_2$ or a simple alternative argument noting that $Z_{L-1 ,N-n}$ decays exponentially for $n\gg (\rho -\rho_c )L$. 
For the second term we get for $u<\rho -\rho_c$ the asymptotics
\[
S_2 \simeq \frac{\theta}{\rho} L^{\kappa +2-\gamma} \int_0^u z^{\kappa +1} \Big(\frac{\rho -\rho_c -z}{\rho -\rho_c} \Big)^\kappa dz\to 0\quad\mbox{as }L\to\infty\ ,
\]
where we used Lemma \ref{zasymp3} and $\gamma >\kappa +2$.

For $u=\rho -\rho_c$ we get the following contribution for any $\epsilon\in (0,1/2)$
\[
S_2 \simeq \frac{\theta L^{\kappa +1-\gamma} (\rho -\rho_c )^{\kappa +1} }{\rho \, Z_{N,L}} \underbrace{\sum_{|n-(\rho -\rho_c )L|<L^{1/2+\epsilon}} Z_{L-1,N-n}}_{\geq w^L (|\sum_x \eta_x -\rho_c L|<L^{1/2+\epsilon} )=1} \simeq \frac{\rho -\rho_c}{\rho}
\]
using Lemma \ref{zasymp3}, the fact that the event $|\sum_x \eta_x -\rho_c L|<L^{1/2+\epsilon}$ is typical for the product measure $w^L$ \eqref{wmoments} and that $S_2 \leq 1- S_1 \simeq \frac{\rho -\rho_c}{\rho}$. Therefore all other contributions to \eqref{toshow3} (including from $S_1$) vanish for $u>\rho-\rho_c$, and $S_3$ vanishes in analogy to the proof of Theorem \ref{cluster} since $S_2$ converges. Therefore
\[
\pi_{L,N}\left[\tilde{\eta}_1 \leq L u\right] \to \begin{cases}
    \rho_c /\rho &,\ u\in (0,\rho -\rho_c )\\
    1 &,\ u\in [\rho -\rho_c ,\rho ]
\end{cases}
\]
and this implies the second statement. So we have convergence in distribution and either $\tilde\eta_1 /L\to 0$ or $\tilde\eta_1 /L\simeq\frac1L\max\limits_{x\in\Lambda} \eta_x \to \rho -\rho_c$. In the first case the statement of the Theorem can be iterated, again based on the product form \eqref{sbformula} of size-biased distributions in analogy to Theorem \ref{cluster}. This implies $\frac1L\max_{x\in\Lambda} \eta_x \overset{d}{\longrightarrow} \rho -\rho_c$.\hfill $\Box$\\

\bibliography{sn-bibliography}
%\printbibliography

\end{document}